\documentclass[a4paper,10pt]{amsart}

\usepackage{amssymb}
\usepackage{fullpage}
\usepackage[OT4]{fontenc}

\theoremstyle{definition}
\newtheorem{definition}{Definition}
\theoremstyle{plain}
\newtheorem{theorem}[definition]{Theorem}
\newtheorem{proposition}[definition]{Proposition}
\newtheorem{lemma}[definition]{Lemma}
\newtheorem{conjecture}[definition]{Conjecture}

\theoremstyle{remark}

\DeclareMathOperator{\Ass}{Ass}
\DeclareMathOperator{\Pic}{Pic}

\def\field{\mathbb{K}}
\def\PP{\mathbb{P}}
\def\seqmul{\overline{m}}

\let\to\longrightarrow
\let\mapsto\longmapsto

\begin{document}

\title{Symbolic powers of ideals of generic points in $\PP^3$}

\author{Marcin Dumnicki}

\dedicatory{
Institute of Mathematics, Jagiellonian University, \\
ul. \L{}ojasiewicza 6, 30-348 Krak\'ow, Poland \\
Email address: Marcin.Dumnicki@im.uj.edu.pl\\
}



\begin{abstract}
B. Harbourne and C. Huneke conjectured that for any ideal $I$ of fat points in $\PP^N$
its $r$-th symbolic power $I^{(r)}$ should be contained in $M^{(N-1)r}I^r$, where
$M$ denotes the homogeneous maximal ideal in the ring of coordinates of $\PP^N$.
We show that this conjecture holds for the ideal of any number of simple (not fat) points
in general position in $\PP^3$ and for at most $N+1$ simple points in general position in $\PP^N$.
As a corollary we give a positive answer to Chudnovsky Conjecture in the case of generic points in $\PP^3$.
\end{abstract}

\maketitle

\section{Introduction}

Let $\field$ be a field of chracteristic zero, let $\field[\PP^N]=\field[x_0,\dots,x_N]$ denote the ring of coordinates of
the projective space with standard grading. Let $I \subset \field[\PP^N]$ be a homogeneous ideal.
By $m$-th symbolic power we define
$$
I^{(m)} = \field[\PP^N] \cap \big( \bigcap_{\mathfrak{p} \in \Ass(I)} I^{m} \field[\PP^N]_{\mathfrak{p}} \big),
$$
where the intersection is taken in the ring of fractions of $\field[\PP^N]$.
By a \emph{fat points ideal} we denote the ideal
$$I = \bigcap_{j=1}^{n} \mathfrak{m}_{p_j}^{m_j},$$
where $\mathfrak{m}_{p}$ denotes the ideal of forms vanishing at a point $p \in \PP^N$, and $m_j \geq 1$ are integers.
Observe that for a fat points ideal $I$ as above
$$I^{(m)} = \bigcap_{j=1}^{n} \mathfrak{m}_{p_j}^{m_j m}.$$
Let $M=(x_0,\dots,x_N) \subset \field[\PP^N]$ be the maximal homogeneous ideal.
In \cite{HaHu} Harbourne and Huneke posed the following Conjecture:

\begin{conjecture}[Harbourne--Huneke]
\label{mainconj}
Let $I$ be any fat points ideal in $\field[\PP^N]$. Then $I^{(rN)} \subset M^{r(N-1)}I^{r}$ holds for all $r \geq 1$.
\end{conjecture}

More about conjectures dealing with containment problems, i.e. showing that $I^{(j)} \subset M^{k}I^{\ell}$ for
various $j,k,\ell$ can be found in \cite{HaHu}, \cite{BoHa1} and \cite{BoHa2}.

A sequence $\seqmul=(m_1,\dots,m_n)$ of $n$ integers will be called a \emph{sequence of multiplicities}.
Define the \emph{ideal of generic fat points} in $\field[\PP^N]$ to be
$$I(\seqmul) = I(m_1,\dots,m_n) = \bigcap_{j=1}^{n} \mathfrak{m}_{p_j}^{m_j}$$
for points $p_1,\dots,p_n$ in general position in $\PP^N$ (for $m_j < 0$ we take $\mathfrak{m}_{p_j}^{m_j} = \field[\PP^N]$).
We will use the following notation:
$$m^{\times s} = \underbrace{(m,\dots,m)}_{s}.$$

In \cite[Proposition 3.10]{HaHu} Conjecture \ref{mainconj} has been verified for ideals of
generic points in $\PP^2$:

\begin{theorem}[Harbourne--Huneke]
Let $I = I(1^{\times n}) \subset \field[\PP^2]$. Then $I^{(2r)} \subset M^rI^r$.
\end{theorem}

In the paper we show that Conjecture \ref{mainconj} holds for any number of generic points in $\PP^3$ and for
at most $N+1$ generic points in $\PP^N$ for all $N \geq 2$:

\begin{theorem}
\label{thm1}
Let $I = I(1^{\times n}) \subset \field[\PP^3]$. Then $I^{(3r)} \subset M^{2r}I^r$.
\end{theorem}

\begin{theorem}
\label{thm2}
Let $I = I(1^{\times n}) \subset \field[\PP^N]$. If $n \leq N+1$ then $I^{(Nr)} \subset M^{(N-1)r}I^r$.
\end{theorem}

Theorem \ref{thm1} follows from Propositions \ref{p3high}, \ref{p3low}, \ref{funda} and \ref{nlow},
while Theorem \ref{thm2} follows from Proposition \ref{funda} and \ref{nlow}.

For a homogeneous ideal $I \subset \field[\PP^N]$ define
$$I_t = \{ f \in I : \deg(f) = t \},$$
and
$$\alpha(I) = \min \{ t \geq 0 : I_t \neq \varnothing \}.$$
In \cite{Ch} Chudnovsky posed the following conjecture:

\begin{conjecture}[Chudnovsky]
\label{chuc}
Let $I$ be the ideal of a finite set of points in $\PP^N$. Then
$$\alpha(I^{(m)}) \geq m \frac{\alpha(I)+N-1}{N}.$$
\end{conjecture}

More about this conjecture can be found e.g. in \cite{HaHu}. It is shown there
that this Conjecture holds for $\PP^2$ and that it holds for an ideal $I \subset \field[\PP^N]$
whenever $I^{(Nr)} \subset M^{(N-1)r}I^r$ (cf. Lemma 3.2 in \cite{HaHu}).
Therefore, as a corollary of Theorem \ref{thm1} we get

\begin{proposition}
Chudnovsky Conjecture holds for any number of generic points in $\PP^3$.
\end{proposition}

We also recall the following numerical quantity, which we will need later:

\begin{proposition}
\label{Wc}
Let $I \subset K[\PP^{N}]$ be a non-zero homogeneous ideal. Define the \emph{Waldschmidt constant}
$$\gamma(I) = \lim_{m \to \infty} \frac{\alpha(I^{(m)})}{m}.$$
Then the following holds:
\begin{enumerate}
\item
$\gamma(I)$ exists and satisfies $\alpha(I^{(m)}) \geq m\gamma(I)$ for all $m \geq 1$,
\item
if $I \subset J$ are ideals then $\gamma(I) \geq \gamma(J)$,
\item
If $I$ is an ideal of fat points then $\gamma(I^{(r)}) = r\gamma(I)$.
\end{enumerate}
\end{proposition}

\begin{proof}
The proof of part 1) is given in \cite{BoHa1} (Lemma 2.3.1 and its proof). For 2) observe
that if $I \subset J$ then $I^{(m)}_t \subset J^{(m)}_t$ hence $\alpha(I^{(m)}) \geq \alpha(J^{(m)})$.
To prove 3) consider the subsequence $rm$, $m \to \infty$:
$$
\gamma(I) = \lim_{m \to \infty} \frac{\alpha(I^{(rm)})}{rm} = \frac{1}{r} \lim_{m \to \infty} \frac{\alpha((I^{(r)})^{(m)})}{m} = \frac{1}{r} \gamma(I^{(r)}).
$$
\end{proof}

\section{Results for $\PP^3$}

In this section we assume $N=3$.

\begin{proposition}
\label{cremona}
Let $\seqmul$ be a sequence of multiplicities. If 
$$I(m_1,m_2,m_3,m_4,\seqmul)_{d} \neq \varnothing$$
then
$$I(m_1+k,m_2+k,m_3+k,m_4+k,\seqmul)_{d+k} \neq \varnothing$$
for $k=2d-m_1-m_2-m_3-m_4$.
\end{proposition}

\begin{proof}
If $m_j+k \geq 0$ for $j=1,2,3,4$ then it follows by applying Cremona transformation \cite[Theorem 3]{Dumalg}.
If $d < m_j$ for some $j$ then
the claim holds trivially. Now consider the case (we can permute points, if necessary) $k+m_4 < 0$.
It follows that 
$$\ell=2d-m_1-m_2-m_3 < 0$$
and each element in $I(m_1,m_2,m_3,m_4,\seqmul)$ has a factor
$f^{\ell}$, where $f$ denotes the equation of the hyperplane passing through points associated with $m_1$, $m_2$, $m_3$
(cf. \cite[Theorem 4]{Dumalg}). It follows that there exists an element of degree $d+\ell$ in 
$$I(m_1+\ell,m_2+\ell,m_3+\ell,m_4,\seqmul).$$
If $d+k=3d-m_1-m_2-m_3-m_4$ is negative then $d+\ell < m_4$ and again the claim holds trivially.
Take $d'=d+\ell$, $m_j'=m_j+\ell$ for $j=1,2,3$, $m_4'=m_4$. By easy computations we check that for 
$$k'=2d'-m_1'-m_2'-m_3'-m_4'$$ we have
$d'+k'=d+k$, $m_j+k=m_j'+k'$ for $j=1,2,3$, $m_4'+k'=0$, $m_4+k < 0$. If $k'+m_j' \geq 0$ for $j=1,2,3$ then we are done by Cremona transformation.
If not, we reorder points and make this trick again, until Cremona is possible.

Alternatively, we can blow up $\pi : X \to \PP^3$ at $p_1,\dots,p_n$ and observe that standard birational transformation of $\PP^3$ induces an action on $\Pic(X)$
such that $\pi^{*}(H) \mapsto 3\pi^{*}H+2(E_1+E_2+E_3+E_4)$, $E_j \mapsto H+(E_1+E_2+E_3+E_4)-E_j$ for $j=1,2,3,4$ and $E_j \mapsto E_j$ for $j > 4$. 
\end{proof} 

\begin{theorem}
\label{glueones}
Let $\seqmul$ be a sequence of multiplicities. Then $\gamma(I(1^{\times 8},\seqmul)) \geq \gamma(I(2,\seqmul))$.
\end{theorem}

\begin{proof}
We begin with showing that $\alpha(I(m^{\times 8})) \geq 2m$. Indeed, assume that there exists an element of degree $2m-1$ in $I(m^{\times 8})$.
We will show that for each $s \geq 0$ there exists an element of degree $2m-(8s^2+1)$ in 
$$I((m-(4s^2-2s))^{\times 4},(m-(4s^2+2s))^{\times 4}),$$
which leads to a contradiction for $s$ big enough.

For $s=0$ the claim holds, now we argue by induction. Take 
$$k=2(2m-(8s^2+1))-4(m-(4s^2-2s))=-8s-2,$$
by Proposition \ref{cremona}
there exists an element of degree $2m-(8s^2+1)+k=2m-8s^2-8s-3$ in 
$$I(m'^{\times 4},(m-(4s^2+2s))^{\times 4})$$
for 
$$m'=m-(4s^2-2s)+k=m-4s^2-6s-2=m-(4(s+1)^2-2(s+1)).$$
Again, take 
$$k=2(2m-8s^2-8s-3)-4(m-(4s^2+2s))=-8s-6,$$
and use Proposition \ref{cremona}.
The degree of the obtained element is equal to 
$$2m-8s^2-8s-3+k=2m-8s^2-16s-8-1=2m-(8(s+1)^2+1)$$
and it belongs to $I(m'^{\times 4},m''^{\times 4})$ for
$$m''=m-4s^2-2s-8s-6=m-(4(s+1)^2+2(s+1)),$$
which completes the induction.

Let $J=I(2,\seqmul)$, let $I=I(1^{\times 8},\seqmul)$, take $m \geq 1$ and assume that $(J^{(m)})_t = \varnothing$. Since $(I(m^{\times 8}))_{2m-1}=\varnothing$,
we can ``glue'' points $m^{\times 8} \to 2m$ as in \cite[Theorem 9]{Dumalg} to show that $(I^{(m)})_t = \varnothing$. Hence
$$\alpha(I^{(m)}) \geq \alpha(J^{(m)})$$
and, by Proposition \ref{Wc},
$$\alpha(J^{(m)}) \geq \gamma(J)m.$$
Dividing by $m$ and passing to the limit completes the proof.
\end{proof}

\begin{proposition}
\label{8k}
Let $s, r \geq 1$, let $n \geq r8^s$. Then
$$\gamma(I(1^{\times n})) \geq 2^s \gamma(I(1^{\times r})).$$
\end{proposition}

\begin{proof}
Since $I(1^{\times n}) \subset I(1^{\times r8^s})$ we have
$\gamma(I(1^{\times n})) \geq \gamma(I(1^{\times r8^s}))$.
Applaying Theorem \ref{glueones} $r8^{s-1}$ times we get
$$\gamma(I(1^{\times r8^s})) \geq \gamma(I(2^{\times r8^{s-1}})).$$
Observe that
$$I(2^{\times r8^{s-1}}) = (I(1^{\times r8^{s-1}}))^{(2)},$$
and hence by Proposition \ref{Wc}
$$\gamma(I(2^{\times r8^{s-1}})) = 2\gamma(I(1^{\times r8^{s-1}})).$$
Repeat the above $s-1$ more times to complete the proof.
\end{proof}

\begin{proposition}
\label{gammas}
The following inequalities hold:
\begin{enumerate}
\item
$\gamma(I(1)) \geq 1$,
\item
$\gamma(I(1^{\times 4})) \geq \frac{4}{3}$,
\item
$\gamma(I(1^{\times n})) \geq \frac{5}{3}$ for $n=5,6,7$,
\item
$\gamma(I(1^{\times 24})) \geq \frac{23}{9}$,
\item
$\gamma(I(2,2,2,1,1)) \geq \frac{8}{3}$,
\item
$\gamma(I(2,2,1,1,1)) \geq \frac{7}{3}$.
\end{enumerate}
\end{proposition}

\begin{proof}
\begin{enumerate}
\item
This is obvious.
\item
Let $I=I(1^{\times 4})$. Take $m \geq 1$ and assume that $(I^{(3m)})_{4m-1} \neq \varnothing$.
By Proposition \ref{cremona} there exists an element of degree $-3$ in $I((-m-2)^{\times 4})=I(0,0,0,0)=\field[\PP^{N}]$, a contradiction.
Hence $\alpha(I^{(3m)}) \geq 4m$. Taking the limit we obtain $\gamma(I) \geq \frac{4}{3}$.
\item
By Proposition \ref{Wc} we have $\gamma(I(1^{\times n})) \geq \gamma(I(1^{\times 5}))$. Assuming $I((3m)^{\times 5})_{5m-1} \neq \varnothing$
gives (by Proposition \ref{cremona}) an element of degree $3m-3$ in $I(3m)$, a contradiction. Hence $\alpha(I(1^{\times 5})^{(3m)}) \geq 5m$ and the claim follows.
\item
By Theorem \ref{glueones}
$$\gamma(I(1^{\times 24})) \geq \gamma(I(2,2,1^{\times 8})).$$
Assume that there exists an element of degree $23m-1$ in $I(18m,18m,(9m)^{\times 8})$. By Proposition \ref{cremona} there exists an element of degree
$15m-3$ in $I((9m)^{\times 6},11m-2,11m-2,m-2,m-2)$. Again, by Proposition \ref{cremona} there exists an element of degree $9m-9$ in
$I(9m,\dots)$, a contradiction. Hence $\alpha(I(2,2,1^{\times 8})^{(9m)}) \geq 23m$ and the claim follows.
\item
If
$$I(6m,6m,6m,3m,3m)_{8m-1} \neq \varnothing$$
then, by Proposition \ref{cremona},
$$I(3m,m-2,m-2,m-2,-2m-2)_{3m-3} \neq \varnothing,$$
a contradiction. As above, the claim follows.
\item
If
$$I(6m,6m,3m,3m,3m)_{7m-1} \neq \varnothing$$
then, by Proposition \ref{cremona},
$$I(3m,2m-2,2m-2,-m-2,-m-2)_{3m-3} \neq \varnothing,$$
a contradiction. As above, the claim follows.
\end{enumerate}
\end{proof}

\begin{proposition}
\label{gamma}
Let $n \geq 1$, let $I=I(1^{\times n})$, let
$$\binom{s}{3} < n \leq \binom{s+1}{3}.$$
If $\gamma(I) \geq \frac{s+1}{3}$
then
$I^{(3r)} \subset M^{2r}I^{r}$ for all $r \geq 1$.
\end{proposition}

\begin{proof}
Let $0 \leq k \leq n$, let $t \geq 1$, observe that $I(1^{\times k})_t$ is a finite dimensional projective vector space over $\field$.
If $f \in I(1^{\times k})_t$ then taking an additional point $p \in \PP^3$ such that $f(p) \neq 0$ we get
$$\dim_{\field} I(1^{\times (k+1)})_t = \dim_{\field} I(1^{\times k})_t - 1.$$
Let $t \geq s-2$. Then
$$\dim_{\field} I_{t} = \dim_{\field} I(1^{\times n})_{t} = \binom{t+3}{3} - n$$
and
$$\dim_{\field} (\field[\PP^3]/I)_{t} = n = \dim_{\field} (\field[\PP^3]/I)_{s-2}.$$
From the above equations we can see that the Hilbert function of $I$ is equal to the Hilbert polynomial of $I$ for degrees
at least $s-2$, hence the Castelnuovo-Mumford regularity of $I$ is at most $s-1$ and $I$ is generated in degrees at most $s-1$.

From Proposition \ref{Wc} and our assumption
$$\frac{\alpha(I^{(3r)})}{3r} \geq \gamma(I) \geq \frac{s+1}{3},$$
thus
$$\alpha(I^{(3r)}) \geq r(s+1) = r(s-1)+r(3-1).$$
By \cite[Proposition 2.3]{HaHu} $I^{(3r)} \subset M^{2r}I^{r}$.
\end{proof}

\begin{theorem}
\label{p3high}
For $n \geq 512$ and $I = I(1^{\times n})$ we have $I^{(3r)} \subset M^{2r}I^{r}$.
\end{theorem}

\begin{proof}
Take $s \geq 1$ such that
$$\binom{s}{3} < n \leq \binom{s+1}{3},$$
let
$$\widetilde{s} = \sqrt[3]{6n}+2.$$
Since
$$s(s-1)(s-2) \geq (s-2)^3$$
we have
$$\sqrt[3]{6n} \geq \sqrt[3]{6\binom{s}{3}} \geq s-2,$$
hence
$$\widetilde{s} \geq s.$$

Consider three cases:
\begin{itemize}
\item
$8^k \leq n \leq 3 \cdot 8^k$ for some $k \geq 0$. By Propositions \ref{Wc}, \ref{8k} and \ref{gammas} we have
$$\gamma(I(1^{\times n})) \geq \gamma(I(1^{\times 8^k})) \geq 2^k \gamma(I(1)) \geq 2^k.$$
By easy computations we check that $k \geq 3$ and
$$1 \geq \frac{\sqrt[3]{18}}{3}+\frac{1}{2^k},$$
hence
$$2^k \geq 2^k\frac{\sqrt[3]{18}}{3}+1 = \frac{\sqrt[3]{6 \cdot 3 \cdot 8^k}+3}{3} \geq \frac{\sqrt[3]{6n}+3}{3} = \frac{\widetilde{s}+1}{3} \geq \frac{s+1}{3},$$
which, by Proposition \ref{gamma}, completes the proof in this case.
\item
$3 \cdot 8^k \leq n \leq 6 \cdot 8^k$ for some $k\geq 0$. By Propositions \ref{Wc}, \ref{8k}, \ref{glueones} and \ref{gammas} we have
$$\gamma(I(1^{\times n})) \geq \gamma(I(1^{\times (3 \cdot 8^k)})) \geq 2^{k-1} \gamma(I(1^{\times 24})) \geq 2^k \frac{23}{18}.$$
By easy computations we check that
$$\frac{23}{18} \geq \frac{\sqrt[3]{36}}{3}+\frac{1}{2^k},$$
hence
$$2^k\frac{23}{18} \geq 2^k\frac{\sqrt[3]{36}}{3}+1 = \frac{\sqrt[3]{6 \cdot 6 \cdot 8^k}+3}{3} \geq \frac{\sqrt[3]{6n}+3}{3} = \frac{\widetilde{s}+1}{3} \geq \frac{s+1}{3},$$
which, by Proposition \ref{gamma}, completes the proof in this case.
\item
$6 \cdot 8^k \leq n \leq 8 \cdot 8^k$ for some $k\geq 0$. By Propositions \ref{Wc}, \ref{8k} and \ref{gammas} we have
$$\gamma(I(1^{\times n})) \geq \gamma(I(1^{\times (6 \cdot 8^k)})) \geq 2^k \gamma(I(1^{\times 6})) \geq 2^k \frac{5}{3}.$$
By easy computations we check that
$$\frac{5}{3} \geq \frac{\sqrt[3]{48}}{3}+\frac{1}{2^k},$$
hence
$$2^k\frac{5}{3} \geq 2^k\frac{\sqrt[3]{46}}{3}+1 = \frac{\sqrt[3]{6 \cdot 8 \cdot 8^k}+3}{3} \geq \frac{\sqrt[3]{6n}+3}{3} = \frac{\widetilde{s}+1}{3} \geq \frac{s+1}{3},$$
which, by Proposition \ref{gamma}, completes the proof in this case.
\end{itemize}
\end{proof}

\begin{theorem}
\label{p3low}
For $5 \leq n \leq 511$ and $I = I(1^{\times n})$ we have $I^{(3r)} \subset M^{2r}I^{r}$.
\end{theorem}

\begin{proof}
Take $s$ satisfying
$$\binom{s}{3} < n \leq \binom{s+1}{3}.$$
We will consider the following cases, in each case using Propositions \ref{Wc}, \ref{8k}, \ref{gammas} and \ref{gamma}.
\begin{itemize}
\item
$s = 15$ or $s=14$, then $n \geq 365$. We have
$$\gamma(I) \geq \gamma(I(1^{\times (5\cdot 64)})) \geq 4\gamma(I(1^{\times 5})) \geq \frac{20}{3} \geq \frac{s+1}{3}.$$
\item
$s=13$ or $s=12$, then $n \geq 221$. We have
$$\gamma(I) \geq \gamma(I(1^{\times (3\cdot 64 + 3 \cdot 8)})) \geq 2\gamma(I(2,2,2,1,1,1)) \geq \frac{16}{3} \geq \frac{s+1}{3}.$$
\item
$s = 11$ or $s=10$ or $s=9$, then $n \geq 85$. We have
$$\gamma(I) \geq \gamma(I(1^{\times 64})) \geq 4\gamma(I(1)) \geq 4 \geq \frac{s+1}{3}.$$
\item
$s = 8$ or $s=7$, then $n \geq 36$. We have
$$\gamma(I) \geq \gamma(I(1^{\times (4\cdot 8)})) \geq 2\gamma(I(1^{\times 4})) \geq \frac{10}{3} \geq \frac{s+1}{3}.$$
\item
$s = 6$, then $n \geq 21$. We have
$$\gamma(I) \geq \gamma(I(1^{\times 20})) \geq \gamma(I(2,2,1,1,1)) \geq \frac{7}{3} \geq \frac{s+1}{3}.$$
\item
$s = 5$, then $n \geq 11$. We have
$$\gamma(I) \geq \gamma(I(1^{\times 8})) \geq 2\gamma(I(1)) \geq 2 \geq \frac{s+1}{3}.$$
\item
$s = 4$ then $n \geq 5$. We have
$$\gamma(I) \geq \gamma(I(1^{\times 5})) \geq \frac{5}{3} \geq \frac{s+1}{3}.$$
\end{itemize}

\end{proof}

\section{Results for $\PP^N$ and $n \leq N+1$}

\begin{proposition}
\label{funda}
Let $1 \leq n \leq N+1$, let $p_1=(1:0:\ldots:0)$, $p_2=(0:1:0:\ldots:0)$, $\ldots$ for $j=1,\dots,n$ be fundamental
points in $\PP^N$, let $I = I(1^{\times n})$, let $J = \cap_{j=1}^{n} \mathfrak{m}_{p_j}$. Let $r \geq 1$.
If $J^{(Nr)} \subset M^{(N-1)r}J^{r}$ then $I^{(Nr)} \subset M^{(N-1)r}I^{r}$.
\end{proposition}

\begin{proof}
Every sequence of $n \leq N+1$ general points in $\PP^{N}$ can be transformed by a linear automorphism into
a sequence of $n$ fundamental points. Since composing with linear forms does not change the degree, the claim follows.
\end{proof}

We will now consider only ideals for a sequence of fundamental points. We can explicitly write down which
elements belongs to such an ideal. For a sequence $(a_0,\dots,a_N)$ of $N+1$ nonnegative integers let
$$x^{(a_0,\dots,a_N)} = x_0^{a_0} \cdot \ldots \cdot x_N^{a_N} \in \field[\PP^3].$$

\begin{proposition}
\label{gener}
Let $I$ be an ideal of $n \leq N+1$ fundamental points, let $m \geq 1$. Then $(I^{(m)})_{t}$ is generated by the following set of monomials:
$$\mathcal{M}=\mathcal{M}(n,m)_t = \bigg\{ x^{(a_0,\dots, a_N)} : \sum_{j=0}^{N} a_{j} = t, \, a_k \leq t-m \text{ for } k=0,\dots,n-1 \bigg\}.$$
\end{proposition}

\begin{proof}
Observe that for $(a_0,\dots, a_N)$ satisfying $\sum_{j=0}^{N} a_j = t$ the following are equivalent for each $k$:
\begin{gather}
a_k \leq t-m,\\
\label{eqm}
\sum_{j=0; j \neq k}^{N} a_j \geq m.
\end{gather}
Observe that
$$\mathfrak{m}_{(1:0:\ldots:0)} = (x_1,\dots,x_N),$$
hence
$$\mathfrak{m}_{(1:0:\ldots:0)}^{m} = \bigg( \big\{ x^{(a_0,\dots, a_N)} : \sum_{j=1}^{N} a_j \geq m \big\} \bigg),$$
which is exactly \eqref{eqm} for $k=0$.
Similar inequalities holds for other fundamental points. The proof is completed by Lemma \ref{monomcap}.
\end{proof}

\begin{lemma}
\label{monomcap}
Let $I, J$ be monomial ideals. Then 
$$I \cap J = ( \{ p \in I \cap J : p \text{ is a monomial}\}).$$
\end{lemma}

\begin{proof}
Let $f \in I \cap J$. We have
$$f = \sum c_j p_j = \sum d_k r_k,$$
where $p_j$'s are monomials from $I$ and $r_k$'s are monomials from $J$. It follows that, after reordering
$r_k$ if necessary, $p_j=r_j$ for all $j$'s and the claim follows.
\end{proof}

\begin{proposition}
\label{nlow}
Let $I$ be the ideal of $n$, $1 \leq n \leq N+1$, fundamental points in $\PP^{N}$, $N \geq 2$. Then $I^{(Nr)} \subset M^{(N-1)r}I^{r}$ for all $r \geq 1$.
\end{proposition}

\begin{proof}
For $n=1$ this is obvious since then $I^{(Nr)} = I^{Nr}$, so let us assume $n \geq 2$.

Choose $t \geq 0$, by Proposition \ref{gener} it is sufficient to show that every element from $\mathcal{M}=\mathcal{M}(n,Nr)_t$ belongs to
$M^{(N-1)r}I^r$. Let $x^{(a_0,\dots,a_N)} \in \mathcal{M}$, of course $t \geq Nr$. Our aim is to show that
$$x^{(a_0,\dots,a_N)} = y \cdot z,$$
where $y$ is a product of $r$ monomials, each of them belonging to $I$,
and $z$ is a monomial of degree at least $(N-1)r$.
Observe that, by Proposition \ref{gener}, $x_j \in I$ for $j \geq n$, while for other indeterminates we have
$$x_jx_k \in I \text{ for } j < n, k < n.$$
Hence $y$ should be equal to the product of $r$ factors, each of them being either a single indeterminate
$x_j$ for $j \geq n$, or a product of two indeterminates $x_jx_k$ for $j,k < n$.
Let
$$s = \sum_{j=0}^{n-1} a_j, \qquad p = \sum_{j=n}^{N} a_j.$$
If $p \geq r$ then $y$ can be taken to be a product of exactly $r$ indeterminates, so $\deg y=r$.
Taking $z=x^{(a_0,\dots,a_N)}/y$ we obtain
$$\deg z = t-r \geq Nr-r = (N-1)r,$$
hence $z \in M^{(N-1)r}$.
Now consider the case where $p<r$. Take $p$ single indeterminates of the form $x_j$ for $j \geq n$ and
$2(r-p)$ indetermines of the form $x_j$ for $j < n$ in such a way that their product $y$ divides $x^{(a_0,\dots,a_N)}$.
It is possible since
$$2(r-p)=2r-2p \leq Nr-p \leq t-p=s.$$
Thus $y \in I^r$, let $z=x^{(a_0,\dots,a_N)}/y$, we have
$$\deg(z)=s-2(r-p).$$
If $t \geq (N+1)r$ then
$$s-2(r-p)=s+p-2r+p \geq t-2r \geq (N+1)r-2r = (N-1)r$$
and consequently $z \in M^{(N-1)r}$.
So now assume that
$$t < (N+1)r.$$
From $n-1 \leq N$ it follows that
$$nN+n-2N-2 \leq nN-N-1$$
and hence
$$t(n-2) \leq (n-2)(N+1)r \leq (nN-N-1)r.$$
The above can be reformulated to
$$2t \geq (N+1)r+n(t-rN).$$
We also know that $x^{(a_0,\dots,a_N)} \in \mathcal{M}$, hence
$$s = \sum_{j=0}^{n-1} a_j \leq n(t-rN)$$
and
$$(N+1)r+n(t-rN) \geq (N+1)r+s.$$
The inequality $2t \geq (N+1)r+s$ can be reformulated to
$$2(p+s)-2r-s\geq (N-1)r$$
which is equivalent to
$$s-2(r-p) \geq (N-1)r,$$
which completes the proof.
\end{proof}

\end{document}